\documentclass[12pt]{huber_article}

\usepackage{amsmath,amssymb,amsthm,array,comment,setspace,url} 
\usepackage{graphicx}
\usepackage{algorithm}
\usepackage{algorithmic}

\newcommand{\prob}{\mathbb{P}}

\newcommand{\bern}{\operatorname{Bern}}

\newcommand{\ind}{{\bf 1}}

\usepackage{fullpage}
\usepackage{amsmath,amsfonts,amsthm}

\newcommand{\on}[1]{{\operatorname{#1}}}
\newcommand{\match}{\mathcal{M}}

\newcommand{\keywords}[1]{\par\addvspace\baselineskip\noindent
  \textbf{Keywords:}\enspace\ignorespaces#1}

\newtheorem{thm}{Theorem}
\newtheorem{lem}{Lemma}

\begin{document}

\title{Reducing the Ising model to matchings}

\maketitle

{\bfseries \sffamily Mark L. Huber} \par
{\slshape Department of Mathematics and Computer Science, 
           Claremont McKenna College} \par
{\slshape mhuber@cmc.edu}
\vskip 1em

{\bfseries \sffamily Jenny Law} \par
{\slshape Department of Mathematics, Duke University} \par
{\slshape waijlaw@math.duke.edu}

\vskip 4em

\begin{abstract}
Canonical paths is one of the most powerful tools available
to show that a Markov chain is rapidly mixing, thereby enabling approximate
sampling from complex high dimensional distributions.  Two success stories for
the canonical paths method are chains for drawing matchings in a graph, and 
a chain for a version of the Ising model called the subgraphs world.  
In this paper, it is shown that a subgraphs world draw can be obtained
by taking a draw from matchings on a graph that is linear in the size
of the original graph.  This provides a partial answer to why canonical paths
works so well for both problems, 
as well as providing a new source of algorithms
for the Ising model. For instance, 
this new reduction immediately yields a fully polynomial
time approximation scheme for the Ising model on a bounded degree graph
when the magnitization is bounded away from 0.
\end{abstract}

\section{Introduction}  

The Markov chain Monte Carlo (MCMC) approach remains the most widely used 
methodology for generating random variates from high dimensional 
distributions.
Let $\pi$ be a distribution on a finite state space $\Omega$.  
A Markov chain is a stochastic process $\{X_1,X_2,\ldots\}$ on $\Omega$
so that $\prob(X_{t+1} \in A|X_1,X_2,\ldots,X_t) = \prob(X_{t+1} \in A|X_t).$

Call a chain on a finite state space 
{\em ergodic} if there exists an $N$ such that for all
$n \geq N$ there is positive probability of traveling from any state $x$
to any other state $y$ in $n$ steps.  For ergodic chains, 
the limiting distribution of $X_t$
will equal the stationary distribution of the chain.  Using well
known methodologies, it is straightforward to build Markov chains whose
stationary distribution matches a target distribution $\pi$.  
(See~\cite{fishman1996} for more details.)

One ingredient is missing:  the question of how large $t$ must be before
the distribution of $X_t$ is close to $\pi$ in some sense such as total
variation.  This $t$ is known as the {\em mixing time} of a Markov chain,
and unless it can be found for the chain in question, MCMC remains a 
heuristic rather than an algorithm for approximate sampling.

A breakthrough occurred when Jerrum and Sinclair developed
the ideas of {\em conductance} and {\em canonical paths} into tools 
capable of proving the mixing
time for complex 
chains on high dimensional spaces.  In \cite{jerrums1989}, they
utilized conductance to show that a chain of Broder~\cite{broder1986} 
for generating uniformly
from the perfect matchings of a graph
was rapidly mixing
under a condition that encompassed a range of interesting 
problems such as uniform generation of regular graphs~\cite{jerrums1990b}.

The development of canonical paths 
for the Ising model~\cite{jerrums1993} followed.
The use of approximate samples derived from a Markov chain 
together with selfreducibility~\cite{jerrumvv1986}
yields an approximation for the partition function of the Ising model,
and this is still the only fully polynomial time 
randomized approximation scheme (fpras) known for this problem.

Later uses of canonical paths included an extension from perfect matchings
to all matchings~\cite{jerrums1996} and an algorithm for finding 
perfect matchings in polynomial time in all graphs 
(\cite{jerrumsv2004},\cite{bezakovasvv2006}).  

Canonical path approaches have 
also been used on such varied problems such as choosing
approximately uniformly from a convex set~\cite{dyerf1991} and 0-1 Knapsack 
solutions~\cite{morriss2004}.

This work is a step towards understanding the relationship between
some of these problems.  We show
\begin{itemize}
  \item{Sampling from the subgraphs Ising model with zero magnitization
        can be accomplished by generating a perfect matching in a graph
        linear in the size of the original graph.}
  \item{Sampling from the subgraphs Ising model with positive magnitization 
        on a graph can be accomplished by
        generating a matching in a graph linear in the 
        size of the original graph.}
  \item{Sampling from perfect matchings in unbalanced bipartite graphs 
        can be reduced to sampling from perfect matchings in balanced bipartite
        graphs.}
  \item{Sampling from matchings in a bipartite graph can be reduced to sampling 
        from perfect matchings in a bipartite graph.}
\end{itemize}
Matchings in unbalanced bipartite graphs arise in the 
approximation of the permanent of 
rectangular matrices,
used in bounding the performance of digital mobile radio systems 
(see~\cite{smithd2001}.)  While conductance could be used directly to 
bound the mixing time of an appropriate chain for this problem, it is easier
to just utilize a sampling reduction.

The remainder of the paper is organized as follows.  Section~\ref{SEC:problems} 
describes the
various distributions that will be studied.  
Section~\ref{SEC:reductions} gives the general form of reduction 
by adding edges and nodes that will be used throughout.
Section~\ref{SEC:ising2matchings}
details the reduction
from the subgraphs world for Ising to the matching distribution.  
Section~\ref{SEC:unbalanced2balanced} 
presents the reduction from unbalanced 
bipartite graphs to balanced bipartite graphs
while Section~\ref{SEC:matchings2perfectmatchings}
gives 
the reduction from matchings to perfect matchings in bipartite graphs.
Finally Section~\ref{SEC:consequences}
notes how this reduction leads to deterministic bounds on
the Ising partition function for large magnitization problems.

\section{The models}
\label{SEC:problems}

This section describes the three models that will  
appear throughout the paper:  The Ising
model, weighted matchings in graphs (the monomer-dimer model in physics), and 
weighted perfect matchings in graphs (the dimer model).  Each of these
three problems is presented in terms of a weight function $w(x)$ where
$x$ is a configuration in state space $\Omega$.  Create a distribution
$\pi$ by setting $\pi(x) = w(x) / Z$, where 
$Z = \sum_{y \in \Omega} w(y)$ is called the {\em partition function}.

\paragraph{The Ising model}
Originally a model of magnetism, the Ising model has a long history of study 
because of the existence of a phase transition on two dimensional lattices 
(see~\cite{simon1993}.)
Three different 
formulations of this model use spins, subgraphs, and random clusters.
While the spins formulation is the most widely known, 
the subgraphs formulation will be of 
most use here (as was also the case in~\cite{jerrums1993}.) 

Given a graph $G = (V,E)$, the state space 
for the subgraphs world is $\Omega_{\on{subs}} = \{0,1\}^E$, 
so that each state, also known as a configuration, indexes a subgraph
of $G$.  Like many distributions of interest, the subgraphs world
is given as a nonnegative weight function over configurations that
is normalized by the partition function.

The parameters for the subgraphs world are as follows.  Each
edge $e$ has weight $\lambda(e) \in [0,1]$ that controls the strength
of interaction between the endpoints of the nodes, and each node $i$
has weight $\mu(i) \in [0,1]$ 
that controls the strength of the magnetic field
for that node.  Given the $\lambda$ and $\mu$ vectors, the weight of 
a configuration is:
\begin{equation}
\label{EQN:weight}
w_\on{subs}(x) = \left[ \prod_{e:x(e) = 1} \lambda(e) \right]
       \left[ \prod_{i:\on{deg}(i) {\textrm{ is odd under }} x} \mu(i) \right].
\end{equation}
Here $\on{deg}(i)$ under $x$ is $\sum_{j:\{i,j\} \in E} x(j)$, and
counts the number of edges with $x(e) = 1$ such that one endpoint
of $e$ is $i$.  As usual, the empty product is taken to be 1.

The vectors $\lambda$ and $\mu$ are in a one to one correspondence with
the inverse temperature and magnetic field parameters found in the more
common spins worlds formulation.  Moreover, the normalizing constant 
for the subgraphs world $Z_\on{subs}$ is a multiple of the normalizing 
constant for the spins world, and the multiple can be calculated explicitly
 (see~\cite{jerrums1993}.)

Calculation of the partition function $Z_\on{subs}$ 
was shown in~\cite{jerrums1993}
to be a $\#P$ complete problem for general graphs, 
and so it is unlikely that a polynomial time
algorithm for finding $Z_\on{subs}$ exactly will be found.

\paragraph{Matchings}

A {\em matching} in a graph is a collection of edges such that no
two edges share an adjacent node.  As before, a collection of edges
can be indexed using $x(e) = 1$ if the edge is in the collection, and 
$x(e) = 0$ if it is not.  Using the weight formulation from before,
$$w_\on{mat}(x) = \left[ \prod_{e:x(e) = 1} \lambda(e) \right]
       \left[ \prod_{i:\on{deg}(i) \geq 2 {\textrm{ under } x}} 
                  0 \right].$$
As with Ising above, the empty product is taken to be 1, and so the
only configurations with positive weight have no nodes with degree
greater than 2.

There exists an algorithm for computing $Z_\on{mat}$ when the graph
$G$ is planar (\cite{fisher1961,kasteleyn1961,temperleyf1961}), but
for general graphs the problem is $\# P$ complete (see~\cite{jerrums1996}.)

\paragraph{Perfect matchings}

A {\em perfect matching} in a general graph is a collection of edges
such that every node is adjacent to exactly one edge.  That is,
$$w_\on{permat}(x) = \left[ \prod_{e:x(e) = 1} \lambda(e) \right]
       \left[ \prod_{i:\on{deg}(i) \neq 1 {\textrm{ under } x}} 
                  0 \right].$$
The partition function is also known as the {\em hafnian} of a matrix
where the $(i,j)$ entry is $\lambda(\{i,j\}).$  

Now consider a bipartite graph $G_B = (V_1 \sqcup V_2,E)$, where
$\#V_1 \leq \#V_2$.  Since the graph is bipartite 
$\{i,j\}$ in $E$ means that exactly one of $i$ and $j$ is in $V_1$,
and exactly one is in $V_2$.  Here a perfect matching will mean that
every node in the smaller partition $V_1$ is adjacent to exactly
one node.  This definition is designed to fit with the definition
of the permanent for rectangular matrices found in~\cite{smithd2001}.
$$w_\on{bipermat}(x) = \left[ \prod_{e:x(e) = 1} \lambda(e) \right]
       \left[ \prod_{i \in V_1:\on{deg}(i) \neq 1 {\textrm{ under } x}} 
                  0 \right].$$
The partition function is also known as the {\em permanent} of 
a matrix where the $(i,j)$ entry is the weight of an edge from the
$i$th node of $V_1$ to the $j$th node of $V_2$.

The permanent problem (and the more general problem of finding the hafnian)
is a $\# P$ complete problem, even under the restrictive condition that
$\lambda(e) \in \{0,1\}$ for all $e \in E$~\cite{jerrums1996}.

\subsection{Canonical paths results}
The presentation here follows that of~\cite{jerrums1996}.  
Recall that the total variation distance between two distributions
$\mu$ and $\pi$ is $\on{dist}_\on{TV}(\mu,\pi) = \sup_A |\mu(A) - \pi(A)|.$
In a canonical paths argument, for any two configurations $x,y$, a path
$x = x_0,x_1,\ldots,x_k = y$ is fixed where 
$\prob(X_{t+1} = x_{i+1}|X_t = x_i) > 0.$  For a set $\Gamma$ of paths,
let $\gamma(x,y)$ denote the path from state $x$ to state $y$.
Given $\Gamma$, the 
parameter $\bar\rho$ measures how often the paths use a particular
move from $x_i$ to $x_{i+1}$.  

A Markov chain is \emph{reversible} with respect to $\pi$ 
if for all states $i$ and $j$,
\[
\pi(i) \prob(X_{t+1} = j|X_t = i) 
 = \pi(j) \prob(X_{t+1} = i|X_t = j).
\]
If a chain is reversible with respect to a distribution $\pi$, 
then $\pi$ must be a stationary distribution.  
For such a reversible chain, let
$Q(e) = Q(\{i,j\}) = \pi(i) \prob(X_{t+1} = j|X_t = i).$  Then
\begin{equation}
\label{EQN:canonicalpaths}
\bar\rho = \bar\rho(\Gamma) := \max_{e} \frac{1}{Q(e)}
  \sum_{x,y:e \in \gamma(x,y)} \pi(x)\pi(y) \on{length}(\gamma(x,y)).
\end{equation}
A relationship between the mixing time of reversible chains 
and canonical paths was shown 
by Diaconis and Stroock~\cite{diaconiss1991}.
The following form of the theorem is from~\cite{jerrums1996}:

\begin{thm}
Let $\cal M$ be a finite, reversible, ergodic Markov chain with
$\prob(X_{t+1} = i|X_t = i) \geq 1/2$ for all $i$, and canonical paths
$\Gamma$.  Fix $x \in \Omega$ and $\epsilon > 0$.  Then
for all 
$t \geq \tau_\epsilon(x,\Gamma) := 
  \bar\rho(\Gamma) (\ln \pi(x)^{-1} + \ln \epsilon^{-1})$ 
and any $A \subseteq \Omega$:
$$|\prob(X_{t} \in A|X_0 = x) - \pi(A)| \leq \epsilon.$$
\end{thm}
The number of steps necessary for the total variation distance to fall
below $\epsilon$ from a starting state 
will be referred to as the \emph{mixing time} of the 
chain.  The theorem states that $\tau_\epsilon(x,\Gamma)$ 
is an upper bound on
the mixing time.  Starting at $x$, this upper bound is proportional
to the loading $\bar\rho$ for the set of canonical paths $\Gamma$.  
Therefore, it is important to find paths that keep the use of any one
edge as low as possible.  

Now the results for the subgraphs world and matching can be stated.
In~\cite{jerrums1993}, it was shown for the subgraphs world that 
\begin{equation*}
\bar\rho(\Gamma) \leq 2 (\#E)^2 / \min_i \mu(i)^4.
\end{equation*}

The configuration with $x(e) = 0$ for all $e$ has weight 1.  All 
configurations have weight at most 1 and there
are at most $2^{\#E}$ configurations, so $\pi(\vec{0}) \geq 2^{-\#E}$,
and starting from the empty
configuration the mixing time is bounded above by
\begin{equation}
\label{EQN:mixtimeIsing}
2 (\# E)^2(\max_i \mu(i)^{-4})[ (\ln 2)\#E + \ln \epsilon^{-1}].
\end{equation}
In~\cite{jerrums1996}, it was shown for the matching problem 
that there exist paths where
\begin{equation*}
\bar\rho(\Gamma) \leq 4 (\#E) (\#V) \lambda'^2, \ \ 
 \lambda' := \max\{1,\lambda(e_1),\ldots,\lambda(e_{\#E})\}
\end{equation*}

It is possible to find the maximum weight matching in polynomial time
via Edmonds algorithm~\cite{edmonds1965}.  There are at most 
$2^{\#E}$ matchings, and so starting from this maximum weight matching
the mixing time is at most
\begin{equation}
\label{EQN:mixtimematchings}
4(\#E) (\#V)\lambda'^2[(\ln 2)\#E + \ln \epsilon^{-1}].
\end{equation}

For perfect matchings in general graphs, no polynomial time algorithm
is known.  However, for bipartite graphs, in a landmark paper 
Jerrum, Sinclair, and Vigoda~\cite{jerrumsv2004} showed how with a
sequence of Markov chains, it was possible to obtain approximately
drawn variates in polynomial time.  The running time of this
procedure was later improved to 
$\Theta(\#V^7 \ln^4 \#V)$~\cite{bezakovasvv2006}.

In the the remainder of the paper, it is shown how to obtain draws from 
the subgraphs world distribution by generating draws from the 
matching distribution.  The running time using the reduction 
will be of the same order as
a direct approach using canonical paths.

\section{Reductions by adding edges and nodes}
\label{SEC:reductions}

All of the reductions presented here have the same form.  
Given graph $G = (V,E)$, consider
drawing from a distribution $\pi$ on $\{0,1\}^E$.  
First construct a new graph $G' = (V',E_1 \cup E_2)$, where
$\phi$ is a one to one correspondence from $E$ to $E_1$.
Then draw $X'$ from $\pi'$ on $G'$, and let
$X(E) = X'(\phi(E)).$  Then the following lemma gives a sufficient
condition for $X$ to be a draw from $\pi$.

\begin{thm}
\label{THM:weights}
Suppose that $\pi(x) = w(x) / Z$ where $Z = \sum_y w(y)$ when used
on $G = (V,E)$, and $\pi'(x) = w'(x) / Z'$ where $Z' = \sum_y w(y)$ over
$G' = (V,E_1 \cup E_2)$, and $\phi$ is a one to one correspondence from
the edges of $E$ to $E_1$.  Suppose that the weight functions satisfy:
\begin{equation}
\label{EQN:weights}
\sum_{x':x'(\phi(E)) = x(E)} w'(x') = w(x) C,
\end{equation}
where $C$ is a constant.  If (\ref{EQN:weights}) holds,
$X'\sim \pi'$ and 
$X(E) = X'(\phi(E))$, then $X \sim \pi$.
\end{thm}

\begin{proof}
Fix $x \in \{0,1\}^E$, and note
$$\prob(X = x) = \sum_{x'} \prob(X = x|X' = x')\prob(X' = x') =
 \sum_{x':x'(\phi(E)) = x(E)} w'(x') / Z'.$$
But by assumption, the numerator of the right hand side is just $C w(x)$, so 
$\prob(X = x) = w(x) C / Z'$, so $X \sim \pi$ and $Z = Z' / C$.
\end{proof}

\paragraph{Reduction for subgraphs to maximum degree 3}
This theorem can be applied to reduce any subgraphs world problem to
one where the maximum degree is three.  Suppose that $i$ is a 
node with degree greater than 3 in $G = (V,E)$.  Let 
$j_1,j_2,\ldots,j_{\on{deg}(i)}$ denote the neighbors of $i$.
Then consider the subgraphs distribution on a new graph $G' = (V',E')$
where $V' = (V \setminus \{i\}) \cup \{i_1,i_2\}$, and 
\begin{eqnarray*}
E' & = & \left(E \cup \{i_1,i_2\}
  \cup \{\{i_1,j_1\},\{i_1,j_2\}\} \cup
 \{\{i_2,j_3\},\{i_2,j_4\},\ldots,\{i_2,j_{\on{deg}(i)}\} \} \right) \\
& &  
 \setminus \cup_{j:\{i,j\} \in E} \{\{i,j\}\}.
\end{eqnarray*}

In other words, node $i$ is being split into two nodes, $i_1$ and $i_2$.
Node $i_1$ is connected to the first two neighbors of $i$, while $i_2$
is connected to the remaining neighbors of $i$.  Finally, the nodes
$i_1$ and $i_2$ are connected.  Figure~\ref{FIG:expansion} illustrates
this process.

\begin{figure}[ht]
  \begin{picture}(400,150)
    \put(100,75){\circle{15}}
    \put(98,72){$i$}
    \thicklines
    \put(94,80){\line(-2,1){30}}
    \put(94,70){\line(-2,-1){30}}

    \put(105,81){\line(2,3){20}}
    \put(107,79){\line(2,1){30}}
    \put(107,72){\line(2,-1){30}}
    \put(105,69){\line(2,-3){20}}

    \put(250,75){\circle{15}}
    \put(247,72){$i_1$}
    \put(244,80){\line(-2,1){30}}
    \put(244,70){\line(-2,-1){30}}

    \put(292,75){\line(-1,0){35}}
    \put(273,80){$1$}
    \put(300,75){\circle{15}}
    \put(297,72){$i_2$}
    \put(305,81){\line(2,3){20}}
    \put(307,79){\line(2,1){30}}
    \put(307,72){\line(2,-1){30}}
    \put(305,69){\line(2,-3){20}}
  \end{picture}
 
  \caption{Splitting a node}
  \label{FIG:expansion}
\end{figure}
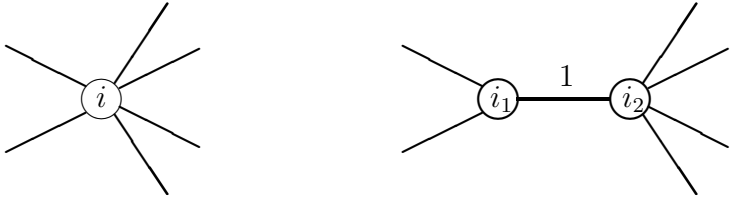

The function $\phi$ is the identity map for any edge that is not adjacent
to $i$.  Edges $\{i,j_1\}$ and $\{i,j_2\}$ map to $\{i_1,j_1\}$ and 
$\{i_1,j_2\}$ respectively, while $\{i,j_k\}$ maps to $\{i_2,j_k\}$ for 
$k \in \{3,\ldots,\on{deg}(i)\}$.

The point of splitting $i$ is that node $i_1$ now has degree 3, and 
$\on{deg}(i_2) = \on{deg}(i) - 1.$  Repeated $\on{deg}(i) - 3$ times,
the nodes that result from splitting $i$ all have degree 3.  This can
be repeated for every node of the graph with degree greater than 3 until
none remain.  The following lemma verifies that this is a valid reduction.

\begin{lem} Suppose $G = (V,E)$ and $G' = (V',E')$ are as described above.
Set $\lambda'(\phi(E)) = \lambda(E),$
$\mu'(V \setminus \{i\}) = \mu(V \setminus\{i\})$,  
$\lambda'(\{i_1,i_2\}) = 1$, and 
$$\mu'(i_1) = \mu'(i_2) = \mu(i)^{-1} - \sqrt{\mu(i)^{-2} - 1},$$
when $\mu(i) > 0$, and $\mu'(i_1) = \mu'(i_2) = 0$ otherwise.
Then drawing $X'$ from $\pi_{\on{subs}}$ on $G'$ and setting
$X(E) = X'(\phi(E))$ yields $X \sim \pi_{\on{subs}}$ on $G$.
\end{lem}

\begin{proof} Fix $x \in \{0,1\}^E$.  
In light of Theorem~\ref{THM:weights}, it suffices to show
\begin{equation}
\label{EQN:verify}
\sum_{x':x'(\phi(E)) = x(E)} w(x') = w(x) C,
\end{equation}
where $w(\cdot) = w_{\on{subs}}(\cdot).$
There is only one edge
in $E'$ not in $\phi(E)$:  the edge $\{i_1,i_2\}$.  There are two 
values for this edge, and so the sum consists of two terms.

The first case to consider is when $i$ has odd degree in $x$.
When $\mu(i) = 0$ the configuration $x$ has probability 0 under 
$\pi_{\on{subs}}$, 
and so no check is necessary.

When $\mu(i) > 0$, then whether 
edge $\{i_1,i_2\}$ is 1 or 0 in $x'$, exactly one of 
$i_1$ and $i_2$ has odd degree, while the other is even.  Hence a factor
of $\mu'(i_1)$ or $\mu'(i_2)$ 
is contributed to the weight, but a factor of 
$\mu(i)$ is missing so $w'(x') = \mu'(i_1) \mu(i)^{-1} w(x)$ for one term
in the sum, and $w'(x') = \mu'(i_2) \mu(i)^{-1} w(x)$ in the other term. 
Hence (\ref{EQN:verify}) becomes:
\begin{equation}
\label{EQN:verify2}
\mu(i)^{-1} \mu'(i_1) w(x) + \mu(i)^{-1} \mu'(i_2) w(x) = w(x) C.
\end{equation}

On the other hand, when $i$ has even degree, 
one choice of $x'(\{i_1,i_2\})$ leads to both $i_1$ and $i_2$ having
odd degree, while with the other choice both have even degree.  Hence
(\ref{EQN:verify}) becomes:
\begin{equation}
\label{EQN:verify3}
w(x) + \mu'(i_1) \mu'(i_2) w(x) = w(x) C.
\end{equation}

When $\mu(i) = 0$, only this equation needs to be satisfied, and
so $C = 1 + \mu'(i_1) \mu'(i_2)$ applies.  However, note that 
\[
1 + \mu'(i_1) \mu'(i_2) 
  = 1 + \left(\mu(i)^{-1} - \sqrt{\mu(i)^{-2} - 1}\right)^2
 = 2\mu(i)^{-2} - 2 \mu(i)^{-1} \sqrt{\mu(i)^{-2} - 1}.
\]
This right hand side is exactly $\mu(i)^{-1}[\mu'(i_1) + \mu'(i_2)],$
so setting $C$ equal to this expression satisfies both
(\ref{EQN:verify2}) and (\ref{EQN:verify3}), finishing the proof.
\end{proof}

The downside of this construction is that the magnetic field is smaller
at each of the duplicated nodes since $\mu'(i_1) < \mu(i)$ unless 
$\mu(i) \in \{0,1\}$, in which case they are equal.  The good news is
that once $\mu(i) \leq \#V^{-1}$ it can be replaced by $\#V^{-1}$ without
changing the distribution too much.  When $\mu(i)$
is this small, the chance that a draw from $\pi_{\on{subs}}$ will result
in all nodes having even degree is at least $\exp(-1),$ so simple
acceptance rejection can be used~\cite{jerrums1993}.

Consider starting with a graph $G = (V,E)$ and repeatedly applying this
reduction until the maximum degree of the graph is 3.  Then the number
of new edges created at each node equals the degree of the node minus
3.  Hence there are at most $\sum_i \on{deg}(i) = 2\#E$ edges created,
resulting in at most $3\#E$ total edges.  

\paragraph{Reduction for subgraphs to all degree 3 nodes}

Not only can the maximum degree be reduced to 3, but the degree 1 and 2
nodes can be eliminated as well.

\begin{lem} Let $G = (V,E)$ be a graph with degree 1 node $i$ and
edge $\{i,j\}$.  
Create $G' = (V',E')$ by $V' = V \setminus \{i\}$,
$E' = E \setminus \{i,j\}$.  Say $X' \sim \pi_{\on{subs}}$ on
$G'$ with parameters $\lambda'(E') = \lambda(E)$, 
$\mu'(V' \setminus \{j\}) = \mu(V' \setminus \{j\})$, and 
$$\mu'(j) =  \frac{\mu(j) + \lambda(\{i,j\}) \mu(i)}
                  {1 + \lambda(\{i,j\})\mu(i)\mu(j)}.$$
Let $X(E') = X'(E')$.  
Given $X(E')$, let $X(\{i,j\}) \sim \bern(\delta)$,
where $\delta$ is $\delta_1$ if $\on{deg}(j)$ is odd under $X'$, and 
$\delta_2$ if $\on{deg}(j)$ is even under $X'$ where
$$\delta_1 = \frac{\lambda(\{i,j\}) \mu(i)}{\mu(j) + \lambda(\{i,j\})\mu(i)}, \ 
  \delta_2 = \frac{\lambda(\{i,j\}) \mu(i) \mu(j)}
             {1 + \lambda(\{i,j\}) \mu(i) \mu(j)}.$$
Then $X \sim \pi_{\on{subs}}$ with parameters $\lambda$ and $\mu$.
\end{lem}

\begin{proof}
Intuitively, a degree 1 node with no magnetic field is irrelevant to the
rest of draw since the edge will always be off.  In this case 
$\delta_1 = \delta_2 = 0$,
so the edge is always removed.
A degree 1 node with
magnetic field is linked only to its nearest neighbor, so it makes
that neighbor a little more likely to be spin up, which is the same as
increasing its magnetic field.  The amount of increase depends on the
magnetic field of $i$ and the strength of interaction to $j$.   

Fix $x \in \{0,1\}^E.$  Let 
$r(x) = \prod_{e \in E':x(e) = 1} \lambda(e)
        \prod_{k \in V' \setminus \{j\}:\on{deg}(k) {\textrm{ is odd in }} x} \mu(k).$
Then there is a unique $x' \in \{0,1\}^{E'}$
that matches $x$ on $E'$, and the following table summarizes the probabilities
needed: \newline
\begin{center}
\begin{tabular}{cccc}
$j$ under $x'$ & $w_{\on{subs}}(x')$ & $x(\{i,j\})$ & $w_{\on{subs}}(x)$ \\
\hline
odd & $r(x) \mu'(j)$ & 1 & $r(x) \lambda(\{i,j\}) \mu(i)$ \\
odd & $r(x) \mu'(j)$ & 0 & $r(x) \mu(j)$ \\
even & $r(x)$ & 1 & $r(x) \lambda(\{i,j\}) \mu(i) \mu(j)$ \\
even & $r(x)$ & 0 & $r(x)$ \\
\end{tabular}
\end{center}

Note $\prob(X = x) = \prob(X'(E') = x(E')) \prob(X(\{i,j\}) = x(\{i,j\})$.
Let $\delta_1 = \prob(X(\{i,j\}) = 1|j \textrm{ is odd in } X')$ and 
$\delta_2 = \prob(X(\{i,j\}) = 1|j \textrm{ is even in } X')$.
Requiring that this probability is proportional to $w_{\on{subs}}(x)$ together
with the above table gives rise to four equations ($C$ is the common constant
of proportionality):
\begin{eqnarray*}
r(x) \mu'(j) \delta_1 & = & r(x) \lambda(\{i,j\}) \mu(i) C \\
r(x) \mu'(j)(1 - \delta_1) & = & r(x) \mu(j) C \\
r(x) \delta_2 & = & r(x) \lambda(\{i,j\}) \mu(i) \mu(j) C \\
r(x) (1 - \delta_2) & = & r(x) C
\end{eqnarray*}
The $r(x)$ factor cancels out from both sides, and adding the last
two equations yields
$C = (1 + \lambda(\{i,j\}) \mu(i) \mu(j))^{-1}$ 
and then $\delta_2 = \lambda(\{i,j\}) \mu(i) \mu(j) C.$

Adding the first two equations yields 
$\mu'(j) = [\mu(j) + \lambda(\{i,j\})\mu(i)]C$, which in turn gives
$\delta_1 = \lambda(\{i,j\}) \mu(i) (\mu(j) + \lambda(\{i,j\}) \mu(i))^{-1}.$
Since $\mu'(i)$, $\delta_1$, and 
$\delta_2$ are in $[0,1]$, this completes the proof.
\end{proof}

Removing degree 2 nodes can be accomplished as well by taking these 
nodes and splitting them into two copies connected by an edge, 
making the degree of each copy equal to 3.  The details are presented
in the appendix.

\section{Reduction of Ising to matchings and perfect matchings}
\label{SEC:ising2matchings}

In this section the technique of the previous section will be used
to reduce the subgraphs world with no magnetic field to 
a perfect matching problem, and to reduce the subgraphs world with positive
magnetic field to a matching problem.

\paragraph{Reduction to perfect matchings when the magnetic field is zero}
When the magnetic field is zero (so $\mu(i) = 0$ for all $i$), the
subgraphs world model can be reduced to a problem where every 
degree is 3 and all $\mu(i)$ remain at 0.  It is from this graph
that the graph for perfect matchings will be constructed.

Each edge $\{i,j\}$ 
in the original graph is given two new nodes $v_{ij}$ and $v_{ji}$.
Connect these two nodes by an edge of weight $\lambda(\{i,j\})$.
Each node $i$ in the original graph is given a new node $d_i$.
Connect $d_i$ and $v_{ij}$ with an edge of weight 1 for all $j$.

Finally, connect nodes of the form $v_{ij}$, $v_{ik}$ for all $j$ and
$k$ with an edge with weight $1/3$.
Set $\phi(\{i,j\}) = \{v_{ij},v_{ji}\}$.
Figure~\ref{FIG:nomaggadget}
illustrates a piece of this transformation.

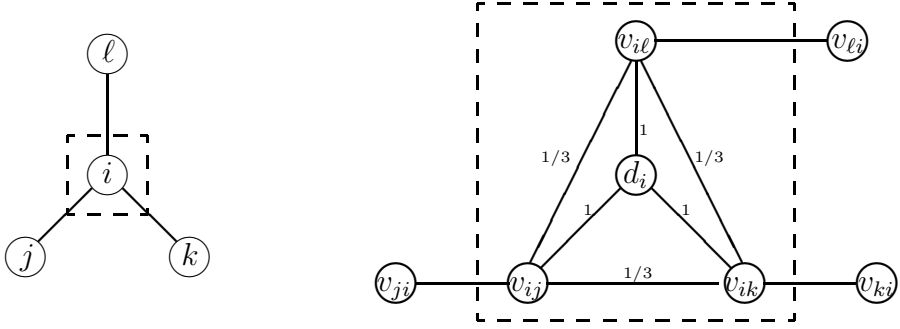
\begin{figure}[ht]
\begin{center}
  \begin{picture}(400,160)(50,0)
    \put(69,44){\circle{15}}
    \put(131,44){\circle{15}}
    \put(100,121){\circle{15}}
    \put(100,75){\circle{15}}
    \put(98,72){$i$}
    \put(67,41){$j$}
    \put(128,41){$k$}
    \put(98,118){$\ell$}
    \thicklines
    \put(106,70){\line(1,-1){20}}
    \put(94,70){\line(-1,-1){20}}
    \put(100,83){\line(0,1){30}}

    \put(85,60){\dashbox{5}(30,30)}

    \put(259,34){\circle{15}}
    \put(252,31){$\small v_{ij}$}
    \put(341,34){\circle{15}}
    \put(333,31){$\small v_{ik}$}
    \put(300,126){\circle{15}}
    \put(293,123){$\small v_{i\ell}$}
    \put(300,75){\circle{15}}
    \put(295,72){$\small d_i$}
    \thicklines
    \put(306,70){\line(1,-1){30}}
    \put(280,60){\tiny 1}
    \put(294,70){\line(-1,-1){30}}
    \put(317,60){\tiny 1}
    \put(300,83){\line(0,1){35}}
    \put(301,90){\tiny 1}

    \put(209,34){\circle{15}}
    \put(202,31){$\small v_{ji}$}
    \put(217,34){\line(1,0){35}}

    \put(391,34){\circle{15}}
    \put(384,31){$\small v_{k i}$}
    \put(348,34){\line(1,0){35}}

    \put(380,126){\circle{15}}
    \put(374,123){$\small v_{\ell i}$}
    \put(307,126){\line(1,0){65}}

    \put(298,118){\line(-1,-2){38}}
    \put(264,80){\tiny $1/3$}
    \put(302,118){\line(1,-2){38}}
    \put(322,80){\tiny $1 / 3$}
    \put(266,34){\line(1,0){65}}
    \put(295,36){\tiny $1 / 3$}

    \put(240,20){\dashbox{5}(120,120)}

  \end{picture}

\end{center}
\caption{New graph for degree 3 nodes}
\label{FIG:nomaggadget}
\end{figure}

For a graph $(V,E)$ where every node has degree 3, 
this transformation creates four nodes for each original node.  These
four nodes are connected by 6 new edges, and the original edges still
remain.  Hence after modification, the new graph has $4\#V$ nodes and 
$\# E + 6\#V$ edges.

\begin{lem}
\label{LEM:ising2perfect}
Given a graph $G = (V,E)$ with all nodes of degree 3,
$\mu(i)$ identically 0 and 
edge weights $\lambda$, let $n(v,1),$ $n(v,2)$ and $n(v,3)$ denote
the neighbors of node $v$.  For each $i$, set 
$A_i = \{d_i,v_{i,n(i,1)},v_{i,n(i,2)},
 v_{i,n(i,3)}\},$ and $E_i$ be all subsets of size 2 in $A_i$ so 
that $(A_i,E_i)$ is the complete graph on 4 vertices.  Set
\begin{eqnarray*}
V' & = & \cup_{i \in V} A_i \\
E' & = & \left( \cup_{\{i,j\} \in E} \{\{v_{ij},v_{ji}\}\} \right)
         \cup \left(\cup_i E_i \right).
\end{eqnarray*}
For edges of the form $e = \{v_{ij},v_{ji}\}$, set 
$\lambda'(e) = \lambda(\{i,j\})$.  For edges $e = \{d_i,v_{ij}\}$, set
$\lambda'(e) = 1$, and for edges $e = \{v_{ij},v_{ik}\}$, set 
$\lambda'(e) = 1/3.$

Then if $X'$ is drawn from $\pi_{\on{permat}}$ on $G'$ with $\lambda'$,
and $X(E) = X'(\phi(E)),$ then $X \sim \pi_{\on{subs}}$ on $G$ with
$\lambda$.
\end{lem}

\begin{proof}  As before, 
let $\phi(\{i,j\}) = \{v_{ij},v_{ji}\}.$
Fix a configuration $x$ in the subgraphs world, and consider the number
of $x'$ such that $x'(\phi(E)) = x(E)$.  The edges that are free in
such an $x'$ are of the form $\{d_i,v_{ij}\}$, 
$\{v_{ij},v_{ik}\}$, or $\{v_{ij},v_{ji}\}$.  
Give weight $1$ to all edges of the
form $\{d_i,v_{ij}\},$ weight $1/3$ to all edges of the form
$\{v_{ij},v_{ik}\},$ and weight $\lambda(\{i,j\})$ to edges
$\{v_{ij},v_{ji}\}$.  

Then for $x$,
the choice of $x$ is a choice for how to fill out $x(E_i)$ for 
each $i$.  For a collection of edges $F$ and a configuration $x'$, let
$p(x'(F))$ be the product of the edge weights over edges with value 1 in $x'$,
that is:
\begin{equation}
p(x'(F)) = \prod_{e \in F:x'(e) = 1} \lambda'(e).
\end{equation}
Let $\match$ be the set of configurations that correspond to a matching,
so $x' \in \match$ says that if 
$e_1$ and $e_2$ are edges that share an endpoint,
then either $x'(e_1)$ or $x'(e_2)$ is 0.
Using this notation
\begin{equation*}
w_{\on{permat}}(x') = p(x'(\phi(E))) \prod_i p(x'(E_i)) \ind(x' \in \match).
\end{equation*}
So
$$\sum_{x':x'(\phi(E)) = x(E)} w_{\on{permat}}(x') 
  = \sum_{x':x'(\phi(E)) = x(E)} p(x'(\phi(E))) 
    \prod_i p(x'(E_i)) \ind(x' \in \match).$$

Since there is no magnetic field, $p(x'(\phi(E))) = w_{\on{subs}}(x)$, and
factoring the sum of the products yields:
$$\sum_{x':x'(\phi(E)) = x(E)} w_{\on{permat}}(x') 
 = w_{\on{subs}}(x) \prod_{i} \sum_{x'(E_i)} p(x'(E_i))
  \ind(x' \in \match).$$

To satisfy Theorem~\ref{THM:weights}, it suffices that 
$\sum_{x'(E_i)} p(x'(E_i)) \ind(x' \in \match) = 1$ 
for all $i$.  There are several cases to consider 
based on the degree of $i$ under $x$.

Suppose first the degree of $i$ under $x$ is 0.  
There are three different $x'(E_i)$ with nonzero weight.
The three different $x'$ come from the fact that $d_i$ 
is matched to one of three different nodes.
Say $x'(\{d_i,v_{ij}\}) = 1$.  Then to match $k$ and $\ell$,
$x'(\{v_{ik},v_{i\ell}\}) = 1$, and all the rest of the edges $e \in E_i$
have $x'(e) = 0$.  Since $\lambda'(\{d_i,v_{ij}) = 1$ and 
$\lambda'(\{v_{ik},v_{i\ell}\}) = 1/3$, $p(x') = 1/3.$

Similarly, $d_i$ could be matched to $v_{ik}$ or $v_{i\ell}$, again
resulting in $p(x'(E_i)) = 1/3$.  Hence
$\sum_{x'(E_i)} p(x'(E_i)) \ind(x' \in \match) = 1/3 + 1/3 + 1/3 = 1.$

Now suppose degree of $i$ under $x$ is 1.
Then $x'(E_i)$ needs to be a perfect matching on $4 - 1 = 3$ 
nodes, which is
impossible.  Similarly, if the degree of $i$ under $x$ is 3, 
$x'(E_i)$ needs to be a perfect matching on 1 node:  also impossible.
Hence in these cases the sum of the weights is 0.  Fortunately,
since there is no magnetic field, $w_{\on{subs}}(x) = 0$ is these
cases also.

Last, suppose that the degree of $i$ in $x$ is $2$.  Then $d_i$ must
be matched in $x'(E_i)$ to whatever node adjacent to $i$ is not 
already matching.  The edge weight of this edge is 1, so 
$\sum_{x'(E_i)} p(x'(E_i)) \ind(x' \in \match) = 1,$ which completes the 
proof.
\end{proof}

Note, a similar gadget appears in~\cite{montroll1964} (pp. 125--147) in
the specific case of the Ising model on two-dimensional lattices.

\paragraph{Reduction to matchings for nonzero magnetic field}
When the magnetic field is positive for every node, the subgraphs world
can be reduced to sampling matchings rather than perfect matchings.
As noted earlier, when $\mu(i) < \#V^{-1}$, changing $\mu(i)$ to 
$\#V^{-1}$
results in a distribution that when sampled from, yields a draw from
the distribution for the original $\mu(i)$ with probability at least
$\exp(-1)$.
Hence any subgraphs model can be altered to have $\mu(i) \geq \#V^{-1}$
for all $i$.

The $G'$ constructed from $G$ remains the same as in the no magnetic
field case, all that changes is the construction of $\lambda'$ for the
matchings.

All edges of the form 
$\{v_{ia},v_{ib}\}$ 
receive the same edge weight $\lambda_1$.
Edges of the form 
$\{d_i,v_{ia}\}$ 
receive
edge weight $\lambda_2$.  But this does not leave enough freedom to 
handle configurations where the degree is either $0$, $1$, $2$, or $3$,
which gives rise to a system of four equations via (\ref{EQN:verify}).
One more parameter is the constant $C$ in the equations, but that still
only gives three unknowns and four equations.

So another parameter $\alpha(i)$ must be added 
to each node.  The new edge weights will be 
$\lambda'(\{v_{ij},v_{ji}\}) = \alpha(i) \alpha(j) \lambda(\{i,j\}),$
and now there are four unknowns.    
The solution to the four equations is presented in the following lemma. 

\begin{lem}
\label{LEM:ising2matchings}
Given graph $G = (V,E)$, with maximum degree 3, build
$G' = (V',E')$ as in Lemma~\ref{LEM:ising2perfect}.
let $G' = (V',E')$ be the same as for the zero
magnetic field case.  
For $i$ of degree 3 in $G$, let $\lambda_2(i)$ be the smallest nonnegative
solution to the 
cubic equation
\begin{equation*}
\mu(i)^2(1 + \lambda_2(i))^3 = 1 + 3 \lambda_2(i) + 3 \lambda_2(i)^2.
\end{equation*}
In fact, there always exists a solution with 
$0 \leq \lambda_2(i) \leq 3 \mu(i)^{-2}.$  Let
\begin{equation*}
\lambda_1(i) = \lambda_2(i)^2(1 + \lambda_2(i))^{-1},\ \ 
  \alpha(i) = \mu(i)(1 + \lambda_2(i)).
\end{equation*}
Then for each edge $e$ of the form $\{v_{ij},v_{ik}\},$ set
$\lambda'(e) = \lambda_1(i)$, while for edges $e$ of the form
$\{d_i,v_{ij}\}$, set $\lambda'(e) = \lambda_2(i).$ 
For edges $e = \{i,j\}$ with $\phi(e) = \{v_{ij},v_{ji}\}$, set
$\lambda'(\phi(e)) = \lambda(e) \alpha(i) \alpha(j).$

Let $X' \sim \pi_{\on{mat}}$, and set $X(E) = X'(\phi(E))$.
Then $X \sim \pi_{\on{subs}}.$
\end{lem}

\begin{proof}
Fix $x \in \{0,1\}^E$.  As before, for $F \subseteq E'$ let
$p(x'(F)) = \prod_{e \in F:x'(e) = 1}\lambda'(e),$ so
\begin{equation*}
\sum_{x':x'(\phi(E)) = x(E)} w_{\on{permat}}(x') 
  = \sum_{x':x'(\phi(E)) = x(E)} p(x'(\phi(E))) \prod_i p(x'(E_i))
   \ind(x' \in \match).
\end{equation*}
Let $d(i,x)$ denote the number of edges adjacent to $i$ that have value
1 in configuration $x$.   
Since each edge $\{v_{ij},v_{ji}\}$ receives an extra factor of 
$\alpha(i)$ and $\alpha(j)$ in its edge weight,
\begin{eqnarray*}
p(x'(\phi(E))) & = & w_\on{subs}(x) \left[ \prod_i \alpha(i)^{d(i,x)}
           \right] \left[
         \prod_{i:d(i,x) \textrm{ is odd}} \mu(i)^{-1} 
         \right] \\
& = & w_{\on{subs}}(x) \prod_i f(i),
\end{eqnarray*}
where 
\[
f(i) := \alpha(i) ^{d(i,x)}
 [\ind(d(i,x) \textrm{ is even}) + 
 \mu(i)^{-1} \ind(d(i,x) \textrm{ is odd})].
\]
Then
\[
\sum_{x':x'(\phi(E)) = x(E)} w_{\on{permat}}(x') 
 = w_{\on{subs}}(x) \prod_i 
 f(i)
  \sum_{x'(E_i)}
  p(x'(E_i)) \ind(x' \in \match).
\]

In order to prove the theorem, it suffices to have each term in the 
product in the right hand side equal a constant for each $i$.  Call
this constant $C(i)$.  Fix $i$.
There are four possible
values for $d(i,x)$:  0, 1, 2, and 3.  Each gives rise to an equation.
In the equations, $\lambda_1$ is the weight for edges between nodes
in $\{v_{ij},v_{ik},v_{i\ell}\}$, and $\lambda_2$
is the weight given to all edges leaving $d_i$.
(For simplicity, 
the dependence of $\lambda_1,\lambda_2,C$ and $\alpha$ on $i$ is
suppressed.)
\begin{eqnarray*}
d(i,x) = 0: \ C & = & [1 + 3\lambda_1 + 3 \lambda_2 
                      + 3 \lambda_1 \lambda_2]\\
d(i,x) = 1: \ C & = & \alpha \mu(i)^{-1} [1 + \lambda_1 
                                               + 2 \lambda_2]  \\
d(i,x) = 2: \ C & = & \alpha^2 [1 + \lambda_2]                          \\
d(i,x) = 3: \ C & = & \alpha^3 \mu(i)^{-1}[1]                              \\
\end{eqnarray*}
The expression in brackets on the right hand is the sum of the weights of 
$x'(E_i)$ with nonzero weight.  When $d(i,x) = 3$, it must be that
$x'(e) = 0$ for all $e \in E_i$, so $p(x'(E_i)) = 1$.

When $d(i,x) = 2$, there are two possible matchings, either 
$d_i$ is matched to the remaining node (giving weight $\lambda_2$) or
not (giving weight $1$).

When $d(i,x) = 1$, one of $\{v_{ij},v_{ik},v_{i\ell}\}$ is matched:
say without loss of generality $v_{i\ell}$ is taken.  Then the remaining
nodes of $E_i$ are $\{d_i,v_{ij},v_{ik}\}$ and they are all connected
to each other.  Hence the possible matchings are:  empty matching (weight 1),
matching $v_{ij}$ to $v_{ik}$ (weight $\lambda_1$),
matching $d_i$ to $v_{ij}$ (weight $\lambda_2$) and matching
$d_i$ to $v_{ik}$ (again weight $\lambda_2$).  Hence the sum of the weights
of the matchings is $1 + \lambda_1 + 2 \lambda_2$.  

When $d(i,x) = 0$, any matching in $E_i$ contributes to the sum.
There is one matching of size 0 (weight 1), six matchings of size 1
(three weight $\lambda_1$, three weight $\lambda_2$), and three matchings
of size 2 (all weigh $\lambda_1 \lambda_2$).  Hence the total sum
of weights is 
$1 + 3 \lambda_1 + 3 \lambda_2 + 3 \lambda_1 \lambda_2.$

Now the solution must be checked.  First show existence.  Let
$$g(x) := \mu(i)^2(1 + x)^3 - (1 + 3x + 3x^2).$$  
Note that 
$$g(3\mu(i)^{-2}) = \mu(i)^2 + 8 + 18 \mu(i)^{-2} > 0,$$
but 
$g(0) = \mu(i)^2 - 1 \leq 0$
Since $g$ is continuous, there must be a solution to $g(x) = \mu(i)^2$
with $x \in [0,3 \mu(i)^{-2}].$

Since $\alpha = \mu(i)(1 + \lambda_2),$ the $d(i,x) = 3$ equation has
$$C = \mu(i)^3(1 + \lambda_2)^3 \mu(i)^{-1} = 1 + 3\lambda_2 + 3\lambda_2^2.$$
Similarly, the $d(i,x) = 2$ equation has
$$C = \mu(i)^2(1 + \lambda_2)^3 = 1 + 3 \lambda_2 + 3\lambda_2^2.$$
For the $d(i,x) = 1$ equation, using $\lambda_1 = \lambda_2^2 / (1 + \lambda_2)$
yields:
$$C = (1 + \lambda_2)(1 + \frac{\lambda_2^2}{1 + \lambda_2} + 2 \lambda_2)
 = 1 + 3 \lambda_2 + 3 \lambda_2^2.$$

Finally, in the $d(i,x) = 0$ equation, 
$$C = 1 + 3 \lambda_1 + 3\lambda_2 + 3\lambda_1\lambda_2 = 
 1 + \frac{3\lambda_2^2}{1 + \lambda_2} + 3\lambda_2 + 
     \frac{3\lambda_2^3}{1 + \lambda_2} 
 = 1 + 3 \lambda_2 + 3 \lambda_2^2.$$

Hence this is a valid solution to all four equations, and by 
Theorem~\ref{THM:weights}, the result follows.

\end{proof}

\section{Reduction of perfect matchings in 
              unbalanced graphs to balanced graphs}
\label{SEC:unbalanced2balanced}

Consider finding the permanent of a bipartite graph
$G = (V_1 \sqcup V_2,E)$
so $(\forall \{i,j\} \in E)(\#(\{i,j\}\cap V_1) = \#(\{i,j\} \cap V_2) = 1).$
Call the graph {\em unbalanced} if $\#V_1 < \#V_2,$ and call a 
configuration $x \in \{0,1\}^E$ a {\em perfect matching in an unbalanced
bipartite graph} if for all $i \in V_1$, there exists a $j \in V_2$ such that 
$\{i,j\} \in E$ and $x(\{i,j\}) = 1.$

Smith studied the number of perfect matchings in unbalanced 
bipartite graphs
to bound the performance of digital mobile radio systems~\cite{smith1998}.
The number of perfect matchings is equal to 
the permanent of the rectangular adjacency matrix (see~\cite{smithd2001}).

The permanent of square matrices has attracted far more study than
the rectangular case, and so
the goal here is to reduce the problem of generating variates from
unbalanced graphs to generating variates from balanced graphs.

Fortunately, the reduction is easy to describe.  Let $V_3$ consist of 
$\#V_2 - \#V_1$ new nodes, and add edges from every node in $V_3$ to
every node in $V_2$, and give them weight 1.  
Then $G$ is still bipartite with node partition 
$(V_1 \sqcup V_3,V_2)$.

\begin{figure}[ht]
\begin{center}
  \begin{picture}(400,150)(-30,0)
    \thicklines
    \put(10,80){\oval(20,40)}
    \put(3,80){$V_1$}
    \put(110,50){\oval(20,100)}
    \put(108,50){$V_2$}

    \put(17,90){\line(6,-1){90}}
    \put(17,80){\line(6,0){90}}
    \put(17,90){\line(6,-1){90}}
    \put(17,65){\line(2,-1){90}}

    \put(160,50){$\Rightarrow$}

    \put(210,80){\oval(20,40)}
    \put(203,80){$V_1$}
    \put(210,27){\oval(20,55)}
    \put(203,27){$V_3$}
    \put(310,50){\oval(20,100)}
    \put(308,50){$V_2$}

    \put(217,90){\line(6,-1){90}}
    \put(217,80){\line(6,0){90}}
    \put(217,90){\line(6,-1){90}}
    \put(217,65){\line(2,-1){90}}

    \put(217,10){\line(1,0){90}}
    \put(217,10){\line(2,1){90}}
    \put(217,25){\line(1,0){90}}
    \put(217,25){\line(2,1){90}}
    \put(217,40){\line(1,0){90}}
    \put(217,40){\line(2,1){90}}

  \end{picture}
\end{center}
\caption{Unbalanced to balanced bipartite graph}
\end{figure}
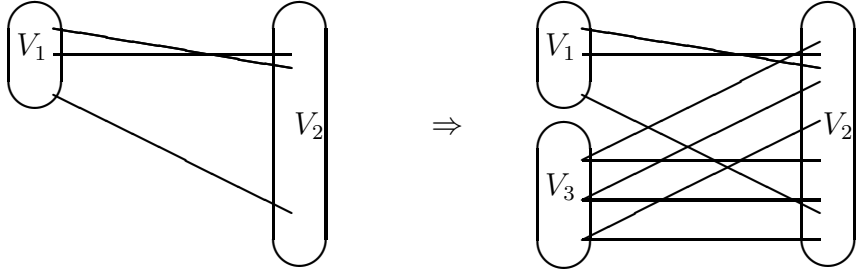

Consider a configuration $x$ that was a perfect matching in the original
graph.  Then if $x'$ is a configuration on the new graph with
$x'(E) = x(E)$, then exactly $\#V_1$ nodes in $V_2$ are already matched.
So to fill out $x'$, the remaining nodes in $V_2$ must be matched to 
$V_3$, and there are exactly $(\#V_2 - \#V_1)!$ ways to accomplish this.
Each of these has weight equal to the weight of $x(E)$, and so 
$\sum_{x':x'(E) = x(E)} w'(x') = w(x) (\#V_2 - \#V_1)!$.  
Theorem~\ref{THM:weights} then says that drawing $x'$ from perfect
matchings on the new graph and keeping $x(E) = x'(E)$ results in
a draw from the perfect matchings on the original unbalanced graph.

Hence any algorithm for simulation and approximation of the partition
function for perfect matchings on balanced graphs 
(such as~\cite{jerrumsv2004},\cite{bezakovasvv2006}) can also be
used for unbalanced problems.

\section{Reducing matchings to perfect matchings in bipartite graphs}
\label{SEC:matchings2perfectmatchings}

Consider a bipartite graph $G = (V_1 \sqcup V_2,E)$.  The problem of 
sampling from all matchings
in this graph can be reduced to sampling perfect matchings as follows.
First, for each node $i \in V_1$, create a node $i'$ and add edge $\{i,i'\}$
with edge weight 1.
If $V_3$ is the set of $i'$, this creates a new, unbalanced bipartite graph
$G' = (V_1 \sqcup (V_2 \cup V_3),E)$.  Furthermore, each matching in the
original graph corresponds to a perfect matching in the new graph with
equal weight.  So a perfect matching sampled from the new graph yields
a matching in the original.

This can be reduced to a balanced
bipartite graph as described in the previous section:  
create $\#V_2$ new nodes and connect each of them to all of the nodes in 
$V_2 \cup V_3$.  The final graph is still bipartite but now is balanced.

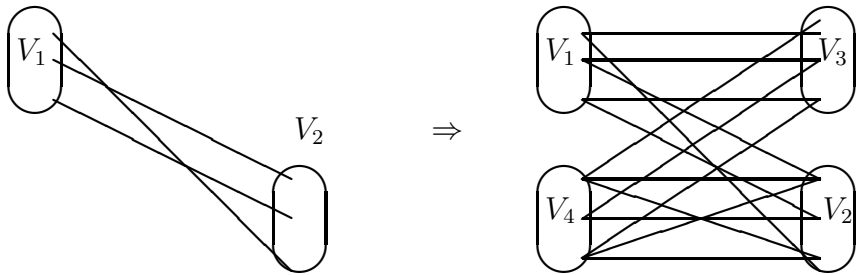
\begin{figure}[ht]
\begin{center}
  \begin{picture}(400,150)(-30,0)
    \thicklines
    \put(10,80){\oval(20,40)}
    \put(3,80){$V_1$}
    \put(110,20){\oval(20,40)}
    \put(108,50){$V_2$}

    \put(17,90){\line(1,-1){90}}
    \put(17,80){\line(2,-1){90}}
    \put(17,65){\line(2,-1){90}}

    \put(160,50){$\Rightarrow$}

    \put(210,80){\oval(20,40)}
    \put(203,80){$V_1$}
    \put(310,20){\oval(20,40)}
    \put(308,20){$V_2$}
    \put(210,20){\oval(20,40)}
    \put(203,20){$V_4$}
    \put(310,80){\oval(20,40)}
    \put(306,80){$V_3$}

    \put(217,90){\line(1,-1){90}}
    \put(217,80){\line(2,-1){90}}
    \put(217,65){\line(2,-1){90}}

    \put(217,90){\line(1,0){90}}
    \put(217,80){\line(1,0){90}}
    \put(217,65){\line(1,0){90}}
    \put(217,05){\line(1,0){90}}
    \put(217,05){\line(3,1){90}}
    \put(217,05){\line(3,2){90}}
    \put(217,20){\line(1,0){90}}
    \put(217,20){\line(3,2){90}}
    \put(217,35){\line(1,0){90}}
    \put(217,35){\line(3,-1){90}}
    \put(217,35){\line(3,2){90}}

  \end{picture}
\end{center}
\caption{Matchings to perfect matchings}
\end{figure}

The point of this is that the new problem is a perfect matchings 
problem on bipartite graphs, and so samples can be generated in 
polynomial time (\cite{jerrums1993,bezakovasvv2006}) for any values of
the edge weights.  With the matching canonical paths approach, the time
to generate a sample depends on the square of the largest edge weight,
with this reduction, this dependence no longer appears.

\section{Consequences of the Ising reduction}
\label{SEC:consequences}

The purpose of any reduction is so that existing methods for one problem
can be immediately applied to the other.  For instance, it was noted
in equation (\ref{EQN:mixtimeIsing}) that the mixing time for the Ising
model on graph $G = (V,E)$ was upper bounded by
$2(\# E)^2(\max_i \mu(i)^{-4})[ (\ln 2)\#E + \ln \epsilon^{-1}]$
using the canonical paths method.

Suppose instead that the Ising model is first reduced to a graph with
maximum degree 3 with at most $3\#E$ edges and $2\#E$ nodes.  
Then the graph
is further altered so that a draw from the matchings distribution yields
a draw from the subgraphs distribution.  The mixing time for the 
matchings distribution on this graph is
$4(3 \#E) (2 \#E)\lambda'^2[(\ln 2)\#E + \ln \epsilon^{-1}],$ 
and from the construction
in Section~\ref{SEC:ising2matchings} $\lambda' \leq 3 \max_i \mu(i)^{-2}.$
So the total mixing time is the same order in $\#E$ as for direct analysis
on the subgraph world, but the constant 
in front is larger by a factor of 108.

The purpose of the reduction is so that any improvements or new algorithms 
for simulating matchings
will automatically translate into new algorithms for the subgraphs world.

As an example of this, Bayati et. al. \cite{bayatigknt2007} have shown 
that for 
matchings in graphs of bounded degree and bounded edge weights how
it is possible to construct a deterministic 
fully polynomial time approximation scheme for computing the partition
function for the set of matchings.  To be precise, their result states:

\begin{thm}
For a graph $G = (V,E)$ with edge weights $\lambda$, and 
$\epsilon > 0$, there exists an $\exp(\epsilon)$ approximation algorithm
for $Z = \sum_{x \in \{0,1\}^E:x \textrm{ a matching}} 
         \prod_{e:x(e) = 1} \lambda(e)$
that runs in time $O(n / \epsilon)^{\kappa \log \Delta + 1}$, where
$\Delta$ is the maximum degree of the graph, $\lambda = \max_e \lambda(e)$ 
and $\kappa = -2 / \log(1 - 2 / [(1 + \lambda \Delta)^{1/2} + 1]).$
\end{thm}

For the subgraphs distribution, 
after the reduction $\Delta = 3$, and $\lambda \leq 3 \max_i \mu(i)^{-2}$.
It can be shown this makes $\kappa$ less than
$3.06 \max_i \mu(i)^{-1}$.  

Given the relationship between the partition function for subgraphs
and matchings
given by Lemma~\ref{LEM:ising2matchings}, and the well-known relationship
between the subgraphs world partition function and the Ising partition
function (see~\cite{jerrums1993}), this gives a fpras
for the partition function of the Ising model for magnitization bounded
away from 0.

\section{Appendix}

There are several ways to deal with degree 2 nodes.  The simplest is
to ``clone'' the node by replacing it with two nodes connected by an
edge of weight 1, as shown in Figure~\ref{FIG:degree2}.

\begin{figure}[ht]
\begin{center}
  \begin{picture}(400,160)(50,0)
    \put(70,70){\circle{15}}
    \put(130,70){\circle{15}}
    \put(190,70){\circle{15}}
    \put(128,68){$i$}
    \put(68,68){$j$}
    \put(188,68){$k$}
    \thicklines
    \put(78,70){\line(1,0){45}}
    \put(138,70){\line(1,0){45}}

    \put(270,70){\circle{15}}
    \put(330,40){\circle{15}}
    \put(330,100){\circle{15}}
    \put(390,70){\circle{15}}
    \put(327,97){$i'$}
    \put(327,37){$i''$}
    \put(268,68){$j$}
    \put(388,68){$k$}
    \thicklines
    \put(277,75){\line(2,1){45}}
    \put(277,65){\line(2,-1){45}}
    \put(338,97){\line(2,-1){45}}
    \put(338,43){\line(2,1){45}}
    \put(329,48){\line(0,1){45}}
    \put(331,68){\tiny 1}
    \put(327,110){\tiny $\mu_1$}
    \put(327,28){\tiny $\mu_1$}

    \put(360,90){\tiny $\lambda_2$}
    \put(360,48){\tiny $\lambda_2$}
    \put(293,90){\tiny $\lambda_1$}
    \put(293,48){\tiny $\lambda_1$}

  \end{picture}

\end{center}
\caption{New graph for degree 2 nodes}
\label{FIG:degree2}
\end{figure}
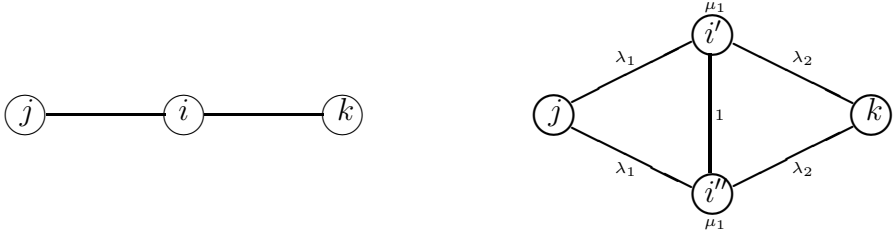

\begin{lem}
Let $G = (V,E)$ be a graph where node $i$ has degree 2.  Let $\{i,j\}$
and $\{i,k\}$ be the edges adjacent to $i$ 
and suppose the parameters for $\pi_{\on{subs}}$ 
are given by $\lambda$ and $\mu$.
Construct $G' = (V',E')$ by letting $V' = (V \setminus \{i\}) \cup \{i',i''\}$
and 
$$E' = (E \setminus \{\{i,j\},\{i,k\}\}) \cup \{ \{i',j\},\{i'',j\},
 \{i',k\},\{i'',k\},\{i',i''\}\}.$$

Set $\lambda_1$, $\lambda_2$, and $\mu_1$ in $[0,1]$ to satisfy
\begin{equation}
\label{EQN:degree2}
\frac{2\lambda_1}{1 + \lambda_1^2} = \lambda(\{i,j\}),\ 
  \frac{2\lambda_2}{1 + \lambda_2^2} = \lambda(\{i,k\}),\
  \frac{2\mu_1}{1 + \mu_1^2} = \mu(i).
\end{equation}
Let $\lambda'(\{i',j\}) = \lambda'(\{i'',j\}) = \lambda_1$,
$\lambda'(\{i',k\}) = \lambda'(\{i'',k\}) = \lambda_2$, 
$\mu'(i') = \mu'(i'') = \mu_1$, and $\lambda'(\{i',i''\}) = 1.$
For all other edges and nodes, let $\lambda'$ and $\mu'$ equal
$\lambda$ and $\mu$ respectively.

Draw $X' \sim \pi_\on{subs}$ using $(\lambda',\mu')$.  Let 
$X(\{i,j\}) = [X'(\{i',j\}) + X'(\{i'',j\})] \operatorname{mod} 2,$
and 
$X(\{i,k\}) = [X'(\{i',k\}) + X'(\{i'',k\})] \operatorname{mod} 2,$
Then $X \sim \pi_\on{subs}$ using $(\lambda,\mu).$
\end{lem}

\begin{proof}
Note that the equations given
in (\ref{EQN:degree2}) can be easily solved 
for $\lambda_1,\lambda_2$ and $\mu_1$ using the quadratic formula.
For an equation of the form $2x / (1 + x^2) = y$, the left hand
side is 0 at $x = 0$, is 1 at $x = 1$, and is continuous.  Hence for 
$y \in [0,1]$, there is a solution $x \in [0,1].$

Now to show that the output $X$ of the procedure has the correct 
probability.  Consider that if $X(\{i,j\}) = 1$, then either 
$X'(\{i',j\}) = 1$ or $X'(\{i'',j\}) = 1$ but not both.  In either
of the two cases, a factor of $\lambda_1$ is contributed, making a 
total contribution of $2 \lambda_1$.  A similar result
holds when $X(\{i,k\}) = 1$.

When $X(\{i,j\}) = 0$, either both $\{i',j\}$ and $\{i'',j\}$ are
0 in $X'$ or both are 1 in $X'$.  The total weight contribution 
is therefore $1 + \lambda_1^2$.

Now consider what happens with $\{i',i''\}$.  In the 
case that $X(\{i,j\}) = X(\{i,k\}) = 1$, one choice of $X'(\{i',i''\})$
leads to both $i'$ and $i''$ being odd, and the other choice leads
to both $i'$ and $i''$ being even.  This makes the total weight of
these contributions $1 + \mu_1^2$.  

This situation also arises when $X(\{i,j\}) = X(\{i,k\}) = 0.$  On the
other hand, when $X(\{i,j\}) \neq X(\{i,k\})$, one choice of 
$X'(\{i',i''\})$ will make one of $\{i',i''\}$ even and the other odd,
while the other choice flips the parity of both $i'$ and $i''$.
Hence the total contribution to weight is $\mu_1 + \mu_1$.  

Let the function $f$ be the transformation that takes $x'$ and constructs
a state $x$.  That is, $f(x')(e) = x'(e)$ for 
all $e \in E \setminus \{\{i,j\},\{i,k\}\}$, 
$f(x')(\{i,j\}) = (x'(\{i',j\}) + x'(\{i'',j\}))\ \on{ mod }\ 2,$ and
$f(x')(\{i,k\}) = (x'(\{i',k\}) + x'(\{i'',k\}))\ \on{ mod }\ 2.$
In the table below, factors from edges and nodes that appear
in both $x'(E)$ and $f(x')(E)$ are the same, and so are neglected.

\begin{center}
\begin{tabular}{cccc}
$x(\{i,j\})$ & $x(\{i,k\})$ & $\sum_{x':f(x') = x} w_{\on{subs}}(x')$ 
     & $w_{\on{subs}}(f(x'))$\\
\hline 
0 & 0 & $(1 + \lambda_1^2) (1 + \lambda_2^2) (1 + \mu_1^2)$ 
 & 1 \\
0 & 1 & $(1 + \lambda_1^2) (\lambda_2 + \lambda_2) (\mu_1 + \mu_1)$ 
 & $\lambda(\{i,k\}) \mu(i)$ \\
1 & 0 & $(\lambda_1 + \lambda_1) (1 + \lambda_2^2) (\mu_1 + \mu_1)$ 
 & $\lambda(\{i,j\}) \mu(i)$ \\
1 & 1 & $(\lambda_1 + \lambda_1) (\lambda_2 + \lambda_2) (1 + \mu_1^2)$
 & $\lambda(\{i,j\}) \lambda(\{i,k\})$ \\
\end{tabular}
\end{center}

By the way $\lambda_1$, $\lambda_2$ and $\mu_1$ are defined, when
$C = (1+\lambda_1^2)(1+\lambda_2^2)(1 + \mu_1^2),$ the equation
$$\sum_{x':f(x') = x} w_{\on{subs}}(x') = C w_{\on{subs}}(x)$$
holds.  As in the proof of Theorem~\ref{THM:weights}, this implies
\[
\prob(X = x) = \sum_{x':f(x') = x} 
  \prob(X' = x') = C w_{\on{subs}}(x) / Z,
\] 
and
hence $X$ has the desired distribution.
\end{proof}
Note that after this reduction, the degree of $i'$ and $i''$ is 3,
while the degree of $j$ and $k$ is increased by 1.  After raising the
degree of the other nodes, they can be split apart as in 
Section~\ref{SEC:reductions} if needed.

\bibliographystyle{plain}

\begin{thebibliography}{10}

\bibitem{bayatigknt2007}
M.~Bayati, D.~Garmarnik, D.~A. Katz, and P.~Tetali.
\newblock Simple deterministic approximation algorithms for counting matchings.
\newblock In {\em Proc. of 39th ACM Symp. on Theory of Computing}, pages
  122--127, 2007.

\bibitem{bezakovasvv2006}
I.~Bez\'akov\'a, D.~Stefankovic, V.~V. Vazirani, and E.~Vigoda.
\newblock Accelerating simulated annealing for the permanent and combinatorial
  counting problems.
\newblock In {\em Proc. 17th ACM-SIAM Sympos. on Discrete Algorithms}, pages
  900--907, 2006.

\bibitem{broder1986}
A.Z. Broder.
\newblock How hard is it to marry at random? ({O}n the approximation of the
  permanent).
\newblock In {\em Proc. 18th ACM Sympos. on Theory of Computing}, pages 50--58,
  1986.

\bibitem{diaconiss1991}
P.~Diaconis and D.~Stroock.
\newblock Geometric bounds for eigenvalues of {M}arkov chains.
\newblock {\em Ann. Appl. Probab.}, 1:36--61, 1991.

\bibitem{dyerf1991}
Martin~E. Dyer and Alan~M. Frieze.
\newblock Computing the volume of a convex body: A case where randomness
  provably helps.
\newblock In B{\'e}la Bollob{\'a}s, editor, {\em Proceedings of AMS Symposium
  on Probabilistic Combinatorics and Its Applications}, volume~44 of {\em
  Proceedings of Symposia in Applied Mathematics}, pages 123--170. American
  Mathematical Society, 1991.

\bibitem{edmonds1965}
J.~Edmonds.
\newblock Paths, trees and flowers.
\newblock {\em Canad. J. Math.}, 17:449--467, 1965.

\bibitem{fisher1961}
M.E. Fisher.
\newblock Statistical mechanics of dimers on a plane lattice.
\newblock {\em Physics Review}, 124:1664--1672, 1961.

\bibitem{fishman1996}
G.~S. Fishman.
\newblock {\em Monte Carlo: concepts, algorithms, and applications}.
\newblock Springer-Verlag, 1996.

\bibitem{jerrums1989}
M.~Jerrum and A.~Sinclair.
\newblock Approximating the permanent.
\newblock {\em J. Comput.}, 18:1149--1178, 1989.

\bibitem{jerrums1990b}
M.~Jerrum and A.~Sinclair.
\newblock Fast uniform generation of regular graphs.
\newblock {\em Theoretical computer Science}, 73:91--100, 1990.

\bibitem{jerrums1993}
M.~Jerrum and A.~Sinclair.
\newblock Polynomial-time approximation algorithms for the {I}sing model.
\newblock {\em SIAM J. Comput.}, 22:1087--1116, 1993.

\bibitem{jerrums1996}
M.~Jerrum and A.~Sinclair.
\newblock {\em The Markov Chain Monte Carlo Method: An Approach to Approximate
  Counting and Integration}.
\newblock PWS, 1996.

\bibitem{jerrumvv1986}
M.~Jerrum, L.~Valiant, and V.~Vazirani.
\newblock Random generation of combinatorial structures from a uniform
  distribution.
\newblock {\em Theoret. Comput. Sci.}, 43:169--188, 1986.

\bibitem{jerrumsv2004}
M.R. Jerrum, A.~Sinclair, and E.~Vigoda.
\newblock A polynomial-time approximation algorithm for the permanent of a
  matrix with nonnegative entries.
\newblock {\em J. of the ACM}, 51(4):671--697, 2004.

\bibitem{kasteleyn1961}
P.W. Kasteleyn.
\newblock The statistics of dimers on a lattice, {I}., the number of dimer
  arrangements on a quadratic lattice.
\newblock {\em Physica}, 27:1664--1672, 1961.

\bibitem{montroll1964}
E.~Montroll.
\newblock Lattice statistics.
\newblock In {\em Applied Combinatorial Mathematics}, U. of California
  Engineering and Physical Sciences Extension Series. John Wiley and Sons,
  Inc., 1964.

\bibitem{morriss2004}
B.~Morris and A.~Sinclair.
\newblock Random walks on truncated cubes and sampling 0-1 knapsack solutions.
\newblock {\em SIAM J.on Comp.}, pages 195--226, 2004.

\bibitem{simon1993}
B.~Simon.
\newblock {\em The Statistical Mechanics of Lattice Gasses}, volume~1.
\newblock Princeton University Press, 1993.

\bibitem{smithd2001}
P.~Smith and B.~Dawkins.
\newblock Estimating the permanent by importance sampling from a finite
  population.
\newblock {\em J. Statist. Comput. Simul.}, 70:197--214, 2001.

\bibitem{smith1998}
P.J. Smith, H.~Gao, and M.V. Clark.
\newblock Performance bounds for {M}{M}{S}{E} linear macrodiversity combining
  in {R}ayleigh fading, additive interference channels.
\newblock Technical Report 98--14, School of Mathematical and Computing
  Sciences, Victoria U. of Wellington, 1998.

\bibitem{temperleyf1961}
H.N.V. Temperley and M.E. Fisher.
\newblock Dimer problem in statistical mechanics--an exact result.
\newblock {\em Philosophical Magazine}, 6:1061--1063,, 1961.

\end{thebibliography}

\end{document}